\documentclass{amsart}

\usepackage{amsmath}
\usepackage{amsfonts}
\usepackage{amssymb}
\usepackage{amscd}
\usepackage{amsthm}
\usepackage{setspace}
\usepackage{graphicx}
\usepackage{hyperref}
\usepackage[shortlabels]{enumitem}

\newtheorem*{mainthm}{Main Theorem}
\everymath{\displaystyle}

\newtheorem{theorem}{Theorem}[section]
\newtheorem{lemma}[theorem]{Lemma}
\newtheorem{corollary}[theorem]{Corollary}
\newtheorem{proposition}[theorem]{Proposition}

\theoremstyle{definition}
\newtheorem{remark}[theorem]{Remark}

\newtheorem{setup}[theorem]{Setup}

\newcommand{\F}{\mathbb{F}}

\newcommand{\s}{\mathbb{S}}
\newcommand{\V}{\mathbb{V}}

\newcommand{\Z}{\mathbb{Z}}
\newcommand{\Hom}{\textrm{Hom}}
\newcommand{\M}{\mathbb{M}}

\newcommand{\m}{\mathfrak{m}}

\newcommand{\cnl}{cnl(\pmb{k})}

\newenvironment{step}
{\proof}
{  
  \endproof   }

\numberwithin{equation}{section}

\begin{document}

\title{The inverse deformation problem}
\author{Timothy Eardley and Jayanta Manoharmayum }
\address{School of Mathematics and Statistics, University of
    Sheffield,
    Sheffield S3 7RH,
    U.K.}
\email{pmp09tae@sheffield.ac.uk}
\email{J.Manoharmayum@sheffield.ac.uk}

\date{}

\begin{abstract}
We show the inverse deformation problem has an affirmative answer: given a complete
local noetherian ring $A$ with finite residue field $\pmb{k}$, we show that there is a
 topologically finitely generated profinite group $\Gamma$ and an
absolutely irreducible continuous representation $\overline \rho:\Gamma\to GL_n(\pmb{k})$ such that $A$ is the universal deformation ring for $\Gamma,\overline\rho$.

\end{abstract}

\maketitle

\section{Introduction}

The purpose of this paper is to answer in the affirmative   the \emph{inverse deformation problem: 
 Can a given complete noetherian local ring with finite residue field be realised as the universal deformation ring of some residual representation?}
The  problem originated from a question of Flach in \cite{chinburg}; the above   formulation is due to  by Bleher, Chinburg and De Smit in \cite{bcd1}.

To begin with we provide a brief, utilitarian introduction to deformation theory following Mazur (see \cite{mazur1}, \cite{mazur2} for details), and in the process we fix notations and conventions which will be in place throughout this article.

In this paper  $\pmb{k}$ is a finite field of characteristic $p$ and  $W:=W(\pmb{k})$ is its Witt ring. We will call a complete noetherian local $W$-algebra with residue field $\pmb{k}$ a \emph{$\cnl$-ring}; the morphisms between $\cnl$-rings are local homomorphisms of $W$-algebras and shall henceforth be referred to as $\cnl$-ring homomorphisms. If $A$ is a $\cnl$-ring  then we write $\m_A$ for its maximal ideal, $\iota_A: W\to A$ for the $W$-algebra structure and set $W_A:=\iota_A(W)$.

Let $\Gamma$ be a topologically finitely generated profinite group and let $\overline\rho :\Gamma \to GL_n(\pmb{k})$ be a  continuous absolutely irreducible representation which shall be called the residual representation. Then Proposition 1 from Section 1.2 of \cite{mazur1} asserts the existence of a universal deformation for $\overline\rho$. By this we mean that there exists a unique $\cnl$-ring $R$ and a continuous representation
$\rho_R: \Gamma\to GL_n(R)$ lifting $\overline\rho$ (that is, $\overline\rho=\rho_R\mod{\m_R}$)
with  the following universal property: if $A$ is a $\cnl$-ring  and
$\rho_A:\Gamma\to GL_n(A)$ is a continuous representation lifting $\overline\rho$, then there is a unique morphism of $\cnl$-rings $\pi :R\to A$  such that $\pi\circ\rho_R$ is strictly equivalent to $\rho_A$.

The main result of this paper is the following theorem which gives a complete answer to the inverse deformation problem.

\begin{mainthm} Let $\pmb{k}$ be a finite field of characteristic $p$. Then every $\cnl$-ring is a universal deformation ring.

More precisely,  let $A$ be a $\cnl$-ring and let $\Gamma:=SL_n(A)$ where
 $n\geq 3$.  If $\pmb{k}= \F_2$, then in addition assume that  $n\geq 5$.
  Let  $\overline\rho :\Gamma\to GL_n(\pmb{k})$ be the reduction modulo maximal ideal of the standard
representation $\rho_A:\Gamma\to  SL_n(A)$. Then $A$, together with $\rho_A$, is the universal deformation ring for $\Gamma$ and $\overline\rho$.
\end{mainthm}

The outline of the rest of the paper is as follows.
In Section \ref{prelim} we present preparatory material concerning the structure $SL_n$ which will be essential to the proof of the main theorem.  The central result is Proposition \ref{main prop} which will allow us to describe the universal deformation in detail.
The rest of the section gives results which are either well known or elementary.
In Section \ref{main thm proof} we provide the proof of the main theorem, formulated as a series of steps, utilising the structural results of the previous section;
finally,  we prove Proposition \ref{main prop}  in Section \ref{prop proof} following a simplified form of the argument used in \cite{manoharmayum}.

We will now give a brief overview of the developments concerning the inverse deformation problem. However, for a more detailed account see \cite{bcd2}. As mentioned earlier, the inverse deformation problem originated from a question by Flach in \cite{chinburg} which asked if it is possible for a universal deformation ring to not be a complete intersection ring. The motivation behind this question  was that up to that point, although there had been many explicit calculations of deformation rings none were not complete intersection rings.  The first example of a universal deformation ring which was not a complete intersection was $W[[t]]/(t^2,2t)$  given in \cite{bc}  (also see \cite{bc2}). This example was greatly generalised to provide a positive answer for all rings of the form $W(\pmb{k})[[t]]/(p^nt,t^2)$ in  \cite{bcd1}. Furthermore, in \cite{bcd2}  a categorisation of all possible pairs $(\Gamma, \overline{\rho})$ which have  $R=W(\pmb{k})[[t]]/(p^nt,t^2)$ as its universal deformation ring was given. Another class of none complete intersection rings were shown by Rainone in \cite{rainone} to be universal deformation rings, namely the rings $\mathbb{Z}_p[[t]]/(p^n,p^mt)$ where  $p>3$ and $n>m$ are  positive integers.

\section{Preliminaries: some properties of $SL_n$}\label{prelim}

We shall now describe certain aspects of the structure of special linear groups over $\cnl$-rings;
these structural properties are ingredients in the proof of the main theorem.
The central result is Proposition \ref{main prop} below which  allows us to assume, in the set-up described in the theorem, that the image of the universal deformation contains $SL_n(W_R)$; this will be used in Section \ref{main thm proof} to specify the images of elements of $\Gamma$ under the universal deformation.

\begin{proposition}\label{main prop}
 Let $(A,\m_A)$ be a $\cnl$-ring. Fix an integer $n\geq 3$. Additionally, we require that  $n\geq 5$  if $\pmb{k}$ is $\F_2$.
Suppose  $G$ is  a closed subgroup
 of $SL_n(A)$ with full residual image i.e. $G\mod{\m_A}=SL_n(\pmb{k})$. Then there exists an  $X\in GL_n(A)$ with $X\equiv I \mod{\m_A}$ such that $XGX^{-1}\supseteq SL_n(W_A)$.
 \end{proposition}

When $\pmb{k}$ is not equal to either $\F_2$ or $\F_3$  or when $\pmb{k}=\F_4$ and $n\neq 3$ this proposition  is covered by the main theorem of \cite{manoharmayum}. The argument in \emph{loc.\ cit.} required certain cohomological properties of $SL_n(\pmb{k})$ which followed from works of Cline, Parshall and Scott in \cite{CPS}, and Quillen in \cite{quillen}. In this paper we shall indicate how the same argument may be recovered in the excluded cases by using
results of Sah \cite{sah1}, \cite{sah2}. In order not to disrupt the flow of proving the main theorem we will defer the proof of Proposition \ref{main prop} until Section \ref{prop proof}.

We shall now record some standard properties of $SL_n$ and of elementary matrices (see \cite{rosenberg} or \cite{srinivas}, for instance). \emph{For the remainder of  this section,   $n\geq 3$ is a fixed integer. All matrices are $n$ by $n$ matrices  unless specified otherwise.}

Let  $A$ be a ring.  If $1\leq i, j\leq n$ with $i\neq j$ and $x\in A$, then $E_{ij}(x)$ denotes the elementary  $n$ by $n$ matrix
which has $1$s along the diagonal, $x$ at the $(i,j)$-th entry and $0$s elsewhere.
In the next well known result, formulated as a proposition for convenience, we recall that the elementary matrices satisfy the Steinberg relations and state an important result about the structure of  $SL_n(A)$.
\begin{proposition}\label{steinberg}  Let  $A$ be a ring. Then the  elementary matrices generate a subgroup of $SL_n(A)$ and satisfy the following Steinberg relations.
\begin{enumerate}[(a)]
\item $E_{ij}(x)E_{ij}(y)=E_{ij}(x +y)$,
\item $[E_{ij}(x),E_{jk}(y)]=E_{ik}(x y)$ if $i\neq k$
\item $[E_{ij}(x),E_{kl}(y)]=1$ if $i\neq l, j\neq k$
\end{enumerate}
Moreover, if $A$ is a  local ring then the elementary matrices $E_{ij}(x)$ with $x\in A$ generate the whole of $SL_n(A)$ and hence $SL_n(A)$ is a perfect group.
\end{proposition}
The second part of the propostion is essentially covered by the discussion following Example 1.6 in \cite{srinivas}. The argument there shows that $SL_n(A)$ is generated by elementary matrices if $n\geq 2$. Perfectness for $n\geq 3$ is well known and follows from standard identities for elementary matrices. The proposition has the following corollary which is useful in determining the image of a deformation.
\begin{corollary}\label{subgroupimage} Let $A$ and $B$ be $\cnl$-rings. Then the image of a group homomorphism $\rho:SL_n(A)\rightarrow GL_n(B)$ is in fact a subgroup of $SL_n(B)$.
\end{corollary}
\begin{proof} By Proposition \ref{steinberg} the image of $\rho$ is generated by the
images of elementary matrices. Given an elementary matrix $E_{ij}(x)$ in $SL_n(A)$,
pick an integer $k$ between $1$ and $n$ distinct from $i$, $j$. The relation
$[E_{ik}(x),E_{kj}(1)]=E_{ij}(x)$ then implies that the determinant of
$\rho(E_{ij}(x))$ must be $1$.
\end{proof}

We will need to conjugate the  elementary matrix $E_{1n}(x)$ to another elementary matrix  $E_{ij}(x)$, and we want to be able to affect the conjugation independently of the ring $A$ in some sense using signed permutation matrices.
For $1\leq r\neq s\leq n$, set
\begin{itemize}
\item   $(rs)\in GL_n(\Z)$ to be the permutation  matrix which differs from the identity only in that its $r$-{th} and $s$-{th} rows have been transposed,
    \item $D_r$ to be the  diagonal matrix that differs from the identity only in its $(r,r)$-th entry which is $-1$.
    \end{itemize}
Note that $(rs)$ and $D_r$ have determinant $-1$. We  now define certain signed permutation matrices  in  $SL_n(\Z)$ as follows: Given $1\leq i,j\leq n$ with $i\neq j$, define $T_{ij}\in SL_n(\Z)$
by
\begin{equation}\label{Tij defn}
T_{ij}:=\begin{cases}
I & \text{if $(i,j)=(1,n)$,}\\
D_2(1n) & \text{if $(i,j)=(n,1)$},\\
D_n(jn) & \text{if $i=1$ and $j\neq n$},\\
D_1(1i) & \text{if $i\neq 1$ and $j=n$},\\
(1i)(nj) & \text{if $i\neq 1$ and $j\neq n$ and $(i,j)\neq (n,1)$}
\end{cases}
\end{equation}
If $X\in GL_n(\Z)$ then its image in $GL_n(A)$ under the unique ring homomorphism $\Z\to A$ will also be denoted by $X$. We then have the following proposition.
\begin{proposition}\label{transpositions} Let $A$ be a ring.
\begin{enumerate}[(i)]
\item Suppose $X\in GL_n(A)$.  Then $XE_{ij}(1)=E_{ij}(1)X$ for all elementary matrices
$E_{ij}(1)$ with $1\leq i< j\leq n$ if and only if $X=\lambda E_{1n}(x)$ for some $\lambda\in A^\times$, $x\in A$.

\item
$T_{ij}E_{1n}(x)T_{ij}^{-1}=E_{ij}(x)$
for all $1\leq i\neq j\leq n$
 and $x \in A$.
 \end{enumerate}
\end{proposition}

We give a brief sketch of the proof. That $\lambda E_{1n}(x)$ commutes with $E_{ij}(1)$, $1\leq i<j\leq n$ is clear from the Steinberg relations (Proposition \ref{steinberg}(c)). For the other direction, if we let $e_{st}$ denote the matrix unit which has a $1$ in the $(s,t)$-th entry and zeros elsewhere, then the relation $E_{ij}(1)X=XE_{ij}(1)$
implies
\[
\sum_{m=1}^{n}x_{jm}e_{im}=\sum_{m=1}^{n}x_{mi}e_{mj}
\]
and desired conclusion follows. The second part is a straightforward calculation which we skip.

\section{Proof of the main theorem}\label{main thm proof}

We recall from the statement of the main theorem that $(A,\m_A)$ is a fixed $\cnl$-ring and $\Gamma:=SL_n(A)$ where  $n\geq 3$ except in  the case when $\pmb{k}=\F_2$ where we exclude $n=4$.
We write $\rho_A:\Gamma\to SL_n(A)$ for the standard representation of $\Gamma$. The residual representation is then taken to be
\[\overline\rho:=\rho_A\mod{\m_A} : \Gamma\to SL_n(\pmb{k}). \]
Note that $\overline\rho$ is clearly surjective.

We will show that $A$ together with $\rho_A$ is the universal deformation ring for $\overline\rho$. For clarity the argument is split into four small steps.

\begin{step}[Step 1] We begin by observing some characteristics of the universal deformation.
Let $(R,\m_R)$ together with $\rho_R:\Gamma\to GL_n(R)$ be the universal deformation ring for
$\overline\rho:\Gamma\to SL_n(\pmb{k})$.
Note that  $\rho_R$ takes values in $SL_n(R)$ by Corollary \ref{subgroupimage}, and that $\rho_R(\Gamma)\mod{\m_R}=SL_n(\pmb{k})$. Therefore, we may invoke Proposition \ref{main prop} and upon replacement of $\rho_R$ with a strictly equivalent representation we may assume that $\rho_R(\Gamma)$  contains a copy of $SL_n(W_R)$. \end{step}

\begin{step}[Step 2] We now note that the unique  $\cnl$-ring homomorphism $\pi:R\to A$ which is associated with $\rho_A$ by the universal property of $R$, i.e. so that $\pi\circ \rho_R$ is strictly equivalent to $\rho_A$, is compatible with $W$-algebra structure morphisms $\iota_A$ and $\iota_R$.
Schematically, the diagram
\begin{equation}\begin{CD}
W@= W\\
@VV{\iota_R}V   @VV{\iota_A}V\\
R@>{\pi}>> A
\end{CD}
\end{equation}
commutes. Let's make the following observations.
\begin{proposition}\label{assertion1}\mbox{}
\begin{enumerate}[(i)]
\item $\rho_R:\Gamma\to SL_n(R) $ is injective and $\pi: \rho_R(\Gamma)\to SL_n(A)$ is an isomorphism.
\item The map $\pi:R\to A$ is surjective.
\item The restriction  $\pi|_{W_R}: W_R\to W_A$ is an isomorphism.
\end{enumerate}
\end{proposition}

\begin{proof} Part (i) follows from the observations that $\pi\circ\rho_R$ is strictly equivalent to
$\rho_A$, and that $\rho_A$ is an isomorphism. Part (ii) is then immediate.

For part (iii), the discussion in step (1) allows us to pick a $\gamma$ in $\Gamma$ with $\rho_R(\gamma)=E_{12}(1)$. Now $\rho_A(\gamma)$ and $E_{12}(\pi(1))$ have the same order (as they are conjugates), and we may conclude that the
restriction  $\pi|_{W_R}: W_R\to W_A$ must be an isomorphism.
\end{proof}

The above assertion allows us to identify $W_R$ and $W_A$. Henceforth, \emph{we will not differentiate between $\iota_R(x)$ and $\iota_A(x)$ for $x\in W$.}
\end{step}

\begin{step}[Step 3] We will now use what we have shown so far to show that with regard
to the group isomorphism $\pi: \rho_R(\Gamma)\to SL_n(A)$ the inverse image of an
elementary matrix in $SL_n(A)$ is an elementary matrix in $SL_n(R)$. This allows us to
construct a $\cnl$-ring homomorphism $A\to R$ which is a section for $\pi:R\to A$.

We first observe the following lemma.
\begin{lemma}\label{lemma} If $x\in A$ then there exist a unique $\lambda_x$ in $R^\times$ and a unique $s(x)$ in $R$ such that the following holds: $\lambda_xE_{1n}(s(x))\in \rho_R(\Gamma)$ and $\pi(\lambda_xE_{1n}(s(x)))=E_{1n}(x)$.
\end{lemma}
\begin{proof} Uniqueness is immediate from Proposition \ref{assertion1}(i).
For existence, let $X\in \rho_R(\Gamma)$ satisfy $\pi(X)=E_{1n}(x)$. Now $E_{1n}(x)$ commutes with the elementary matrices $E_{ij}(1)$ where $1\leq i<j\leq n$. Then by Proposition \ref{assertion1} and our identification of $W_R$ with $W_A$, the elementary matrices  $E_{ij}(1)$  with $1\leq i<j\leq n$ are in $\rho_R(\Gamma)$ and commute with $X$. Hence by Proposition \ref{transpositions}
we must have $X=\lambda_x E_{1n}(s(x))$ for some $s(x)$ in $ R$  and $\lambda_x $ in $R^{\times}$.
\end{proof}

Now let $x\in A$ and let $s(x)$, $\lambda_x$ in $R$ be as in Lemma \ref{lemma}. From the preceding two steps, the signed
permutation matrices $T_{ij}$s, as defined by the relations (\ref{Tij defn}), are matrices in $\rho_R(\Gamma)$. Since
$\lambda_xE_{ij}(s(x)))=T_{ij}\lambda_xE_{1n}(s(x))T_{ij}^{-1}$ by Proposition \ref{transpositions}, we see that $\lambda_x E_{ij}(s(x))$ is in $\rho_R(\Gamma)$ and is the unique pre-image of
$E_{ij}(x)$ for any $1\leq i\neq j\leq n$. Note that if
$x\in W_A$ then $\lambda_x=1$ and $s(x)=x$ (as we are identifying $W_A$ and $W_R$).

We will now show that $\lambda_x$ is in fact $1$.  Let $i$, $j$, $k$ be three distinct integers in $\{1,2,\ldots ,n\}$.
By considering their inverse images in $\rho_R(\Gamma)$, the relation $E_{ij}(x)= E_{ik}(x)E_{kj}(1)E_{ik}(x)^{-1}E_{kj}(1)^{-1}$ then implies that
\begin{align*}
\lambda_xE_{ij}(s(x))
&=\lambda_xE_{ik}(s(x))E_{kj}(1)\lambda_x^{-1}E_{ik}(s(x)))^{-1}E_{kj}(1)^{-1}\\
&=E_{ij}(s(x)),
\end{align*}
and hence $\lambda_x=1$.

We can now  define the desired section to $\pi : R\to A$.
\begin{proposition}\label{defn of map s}
The function $s: A\to R$ characterised by the following property is well defined:
\begin{quote} If $x\in A$ then  $s(x)$ is the unique element  in $R$ such that $\pi(s(x))=x$
and  the elementary matrix $E_{ij}(s(x))$ is a matrix in $\rho_R(\Gamma)$ for all $1\leq i\neq j\leq n$.\end{quote}
Moreover, the map $s:A\to R$ is in fact a $\cnl$-ring homomorphism.
\end{proposition}
\end{step}

\begin{proof} We have already covered the characterising property defining $s:A\to R$ in Lemma \ref{lemma}
 and the discussion following it.

 We shall now show that the map $s:A\to R$ is a $\cnl$-ring homomorphism. It follows immediately from the construction and Proposition \ref{steinberg}(a) that $s(x+y)=s(x)+s(y)$ for all $x,y\in A$, that $s|_{W_A}$ is the inverse to $\pi|_{W_R}$, and that
$\pi\circ s$ is the identity on $A$. If $1\leq i,j,k\leq n$ are three distinct integers, then the commutation relation $[E_{ij}(s(x)),E_{jk}(s(y))]=E_{ik}(s(x)s(y))$ shows that $s(xy)=s(x)s(y)$. Thus $s:A\to R$ is a ring homomorphism. By construction $s(\m_A)\subseteq \m_R$ and $s$ induces the identity on $A/\m_A=R/\m_R=\pmb{k}$.
\end{proof}

\begin{step}[Step 4] To complete  the proof of the theorem, we have to verify that $\rho_A: \Gamma\to SL_n(A)$ is the universal deformation. Since  the elementary matrices $E_{ij}(x)$ generate $SL_n(A)$ by Proposition \ref{steinberg}, we deduce that the elementary matrices $E_{ij}(s(x))$ generate $\rho_R(\Gamma)$. As $\pi \circ s$ is the identity on $A$, we can now conclude that $s\circ\pi\circ\rho_R= \rho_R$.
By universality, the homomorphism $s\circ\pi:R\to R$ must be the identity on $R$. Thus $\pi: R\to A$ is an isomorphism with inverse $s:A\to R$ and $\pi\circ \rho_R$, resp. $s\circ \rho_A$, is strictly equivalent to $\rho_A$, resp. $\rho_R$. The theorem follows.
\end{step}

\section{Proof of  Proposition \ref{main prop}}\label{prop proof}

As discussed in Section \ref{prelim} the key result, Proposition \ref{main prop}, follows from the main theorem of \cite{manoharmayum} if $\pmb{k}$ has cardinality at least $5$, or $\pmb{k}=\F_4$ and $n\geq 4$. Therefore, we will now investigate the other cases. To this end, we assume the following setup.
\begin{setup}\label{setup} Throughout this section:
\begin{itemize}
\item The finite field $\pmb{k}$ is either $\F_2$ or $\F_3$ or $\F_4$, and   $p$ denotes its characteristic. We set $W_m:=W(\pmb{k})/p^m$.
\item $n$ is a fixed integer subject to the following condition:
if $\pmb{k}=\F_2$ then  $n\geq 5$,
if $\pmb{k}=\F_3$ then  $n\geq 3$, and
if $\pmb{k}=\F_4$ then $n=3$.

\item $\M$, resp. $\M_0$, denotes
the space of $n$ by $n$  matrices over $\pmb{k}$, resp. the space of $n$ by $n$  matrices over $\pmb{k}$ with trace $0$. When $p|n$, we set $\s:=\pmb{k}I$ and $\V=\M_0/\s$.
\end{itemize}
\end{setup}

We remark that if $(A,\m_A)$ is a
$\cnl$-ring then $GL_n(A)$ acts on $\M$ and $\M_0$
by conjugation. We  make free use of standard results on group extensions and cohomology (see \cite{brown}, \cite{NSW}); section 2 of \cite{manoharmayum} covers what will be needed.

The following proposition gathers the various  properties of $SL_n(W_m)$
that will be needed in the proof of Proposition \ref{main prop}.
\begin{proposition}\label{SL_n(W_m)} With the above set up, we have:
\begin{enumerate}[(i),leftmargin=*]
\item \label{fact1} If $p\nmid n$ then  $\M_0$ is an irreducible $SL_n(\pmb{k})$-module; if $p|n$ then $\s$ is the unique non-trivial $SL_n(\pmb{k})$-submodule of $\M_0$. Moreover,
\[\Hom_{SL_n(\pmb{k})}(\M_0,\M_0)\cong \pmb{k}\cong \Hom_{SL_n(\pmb{k})}(\V,\V).\]

\item  \label{fact2} Let $\Gamma_m:=\ker(SL_n(W_{m+1})\xrightarrow{\mod{p^m}}SL_n(W_m))$. Then the extension
\[ I\to\Gamma_m\to  SL_n(W_{m+1})\to SL_n(W_m)\to I \]
does not split.

\item \label{fact3} Suppose $p|n$. Then $H^1(SL_n(\pmb{k}), \pmb{k})$ and  $H^2(SL_n(\pmb{k}), \pmb{k})$ are both $(0)$. Furthermore $H^1(SL_n(W_m),\pmb{k})=(0)$ for all $m\geq 1$.

\item\label{fact4} The inflation map
$H^1(SL_n(W/p^m), \M_0)\to H^1(SL_n(W/p^{m+1}),\M_0)$
 is an isomorphism. Consequently \[ H^1(SL_n(W/p^m), \M_0)\cong  H^1(SL_n(\pmb{k}), \M_0)
=\begin{cases}
(0) & \text{if}\quad p\nmid n,   \\
\pmb{k}&\text{if}\quad p| n.
\end{cases}
\]

\item \label{fact5} Suppose now $p|n$.  Then
\begin{enumerate}[(a)]
\item  If $Z_m$ denotes the
subgroup of scalar matrices in $\Gamma_m$, then the extension
\begin{equation}\label{extn1}
 I\xrightarrow{}\Gamma_m/Z_m\xrightarrow{} SL_n(W_{m+1})/Z_m\xrightarrow{\mod{p^m}} SL_n(W_m)
\xrightarrow{}I.
\end{equation} does not split.
\item The inflation map
$H^1(SL_n(W_{m}),\V)\to H^1(SL_n(W_{m+1}),\V)$ is an isomorphism.
\item The map $H^2(SL_n(W_m),\s)\to  H^2(SL_n(W_m),\M_0)$ induced by the inclusion
$\s\subset\M_0$
is an injection.
\end{enumerate}
\end{enumerate}
\end{proposition}

\begin{proof}
\emph{Part (i):} See Lemma 3.3 of \cite{manoharmayum}.

\medskip
\noindent\emph{Part (ii):} When $m\geq 2$, the non-splitting is covered by the argument in Proposition 3.7 of \cite{manoharmayum}. (See the paragraph around the displayed relation (3.5) there, \emph{loc. cit.})
The case when $\pmb{k}=\F_2$ or $\F_3$, and $m=1$   is Theorem II.7 of \cite{sah1}.

Suppose now $\pmb{k}=\F_4$ and $m=1$. If $\Gamma$ denotes the kernel of the reduction map $GL_3(W/p^2)\to GL_3(\pmb{k})$, then the sequence
\[ I\to \Gamma\to GL_3(W/p^2)\to GL_3(\pmb{k})\to I \]
does not  split by Proposition 0.3 of \cite{sah2}. Let $\widetilde{G}$ be the subgroup of $GL_3(W/p^2)$ consisting of matrices with determinant $1$ modulo $p$. The injectivity of the  restriction map
\[ H^2(GL_3(\pmb{k}),\M)\to H^2(SL_3(\pmb{k}),\M)\]
then  implies that $ I\to \Gamma\to \widetilde{G}\to SL_3(\pmb{k})\to I$ is non-split. Consequently
$I\to \Gamma_1\to SL_3(W/p^2)\to SL_3(\pmb{k})\to I$  can not be split.

\medskip\noindent\emph{Part (iii):} See Proposition III.7 of \cite{sah1} for the first part. Now \[ H^1(\Gamma_m,\pmb{k})^{SL_n(\pmb{k})}\cong \Hom_{SL_n(\pmb{k})}(\M_0,\pmb{k})=(0)\]
by part \ref{fact1}  above. Inflation--restriction then implies that
$H^1(SL_n(W_m),\pmb{k})=(0)$ for all $m\geq 1$.

\medskip\noindent\emph{Part (iv):} Set $\phi :\Gamma_m\to \M_0(\pmb{k})$  to be the identification given by
$\phi(I+p^mM):=M\mod{p}$.
The transgression map
\[
\delta: H^1(\Gamma_m, \M_0)^{SL_n(W/p^m)}\to H^2(SL_n(W/p^m),\M_0(\pmb{k}))
\]
sends $-\phi$ to
the class of the extension
\[
0\to \M_0(\pmb{k})\xrightarrow{\phi^{-1}} SL_n(W/p^{m+1})\to SL_n(W/p^m)\to 1\] in
$ H^2(SL_n(W/p^m),\M_0(\pmb{k}))$ (see Proposition 2.1 of \cite{manoharmayum}).  Since $H^1(\Gamma_m, \M_0)^{SL_n(W/p^m)}$ has dimension $1$ as
a $\pmb{k}$-vector space by part \ref{fact1}, and  $\delta(-\phi)\neq 0$ as the above extension is non-split by part \ref{fact2}, the transgression map $\delta$ is injective and the claim follows.

The group $H^1(SL_n(W/p^m), \M_0)$ is completely determined by $H^1(SL_n(\pmb{k}), \M_0)$. Now $H^1(SL_3(\F_4),\M_0)=(0)$ is covered by \cite{CPS}. For the case when $\pmb{k}=\F_2$ or $\F_3$, first note that  $H^1(SL_n(\pmb{k}),\M)=(0)$ by Theorem III.5 of \cite{sah1}. When $p\nmid n$ the direct sum decomposition $\M=\M_0\oplus \pmb{k}I$ implies $H^1(SL_n(\pmb{k}),\M)=(0)$; if $p|n$
the exact sequence
 $ 0\to \M_0\to \M\to \pmb{k}\to 0$ along with part \ref{fact3} implies that the
 connecting map $H^0(SL_n(\pmb{k}),\pmb{k})\to H^1(SL_n(\pmb{k}),\M_0)$ is an isomorphism.

\medskip\noindent\emph{Part (v):} We give a brief sketch; see    Section 3.3 of
\cite{manoharmayum} for details. Sub-parts (a) and (b) are equivalent (for a fixed $m$) by an argument similar  to part \ref{fact4} and using
\[
H^1(\Gamma/Z,\V)^{SL_n(W_{m})}\cong \Hom_{SL_n(W_{m})}(\V,\V)\cong \pmb{k}.\]
 The splitting of the extension (\ref{extn1}) when $m\geq 2$ is then covered by Lemma 3.9 of
 \cite{manoharmayum}.

 For $m=1$ consider      \begin{equation*}\begin{CD}
H^1(\Gamma_1,\M_0)^{SL_n(\pmb{k})}
@>{\delta}>>  H^2(SL_n(\pmb{k}),\M_0) \\
@VVV   @VVV        \\
H^1(\Gamma_1,\V)^{SL_n(\pmb{k})}
@>{\delta}>>  H^2(SL_n(\pmb{k}),\V)
\end{CD}\end{equation*} where the vertical maps come from
$0\to\s\to\M_0\to \V\to 0$.
The left hand arrow
is an isomorphism by part \ref{fact1}, the right hand arrow is an isomorphism by part \ref{fact3}, and the top arrow is an isomorphism by part \ref{fact4}.  Hence the bottom arrow is also an injection and therefore the inflation map
$H^1(SL_n(\pmb{k}),\V)\to H^1(SL_n(W(\pmb{k})/p^2),\V)$ is an isomorphism.

Now for (c). Note that (b) together part \ref{fact2} gives
\[
H^1(SL_n(W_m),\V)\cong H^1(SL_n(\pmb{k}),\V)\cong H^1(SL_n(\pmb{k}),\M_0)\cong \pmb{k}\]
for all $m\geq 1$.
The exact sequence $0\to \s\to \M_0\to \V\to 0$, together with parts \ref{fact2}, \ref{fact4} and (b) above, now implies that
\[ H^2(SL_n(W_m),\s)\to  H^2(SL_n(W_m),\M_0)\]
is an injection.
\end{proof}

We recall that for an artinian $\cnl$-ring the annihilator of its maximal ideal is non-trivial. As a $\cnl$-ring is the inductive limit of its artinian quotients  Proposition \ref{main prop}  follows from the following special case.

\begin{proposition} We continue with the set up of (\ref{setup}). Let $(A,\m_A)$ be an artinian $\cnl$-ring and let $G$ be a subgroup of $SL_n(A)$. Assume that
$G \mod{t}=SL_n(W_{A/(t)})$ where $t\in A$ is non-zero and satisfies $t\m_A=0$.  Then there is an $X\in SL_n(A)$ with
 $X\equiv I\pmod{t}$ such that $ SL_n(W_A)\subseteq XGX^{-1}$.
\end{proposition}

\begin{proof} We set $B:=A/(t)$ and  $\pi: A\to B$ to be reduction modulo $t$. Then we have an exact sequence
\begin{equation}
0\to \M_0\xrightarrow{\varepsilon} SL_n(A)\xrightarrow{\pi}SL_n(B)\to 1
\end{equation} where the map $\varepsilon: \M_0\to SL_n(A)$ is constructed as follows: Lift $x\in \M_0$ to an $n$ by $n$ matrix $\tilde{x}$ over $A$ and take $\varepsilon(x):=I+t\tilde{x}$. Denote by $\widetilde{G}$ the pre-image of $SL_n(W_B)$ in $SL_n(A)$. Thus
\begin{equation}\label{seq 1}
0\to \M_0\xrightarrow{\varepsilon}\widetilde{G}\xrightarrow{\pi}SL_n(W_B)\to 1
\end{equation}
is exact and $G$, $SL_n(W_A)$ are subgroups of $\widetilde{G}$. There are then the following three possibilities  to consider:
\begin{itemize}
\item $G=\widetilde{G}$, in which chase there is nothing to prove;
\item $\pi:G\to SL_n(W_B)$ is an isomorphism; or,
\item  $G$ fits into an exact sequence $0\to\s\to G\to SL_n(W_B)\to I$.
\end{itemize}

Suppose  $\pi:G\to SL_n(W_B)$ is an isomorphism. Then the sequence (\ref{seq 1}) splits.
Consequently  $\pi: SL_n(W_A)\to SL_n(W_B)$ must also be an isomorphism and $G$ is a
twist of $SL_n(W_A)$ by an element of $H^1(SL_n(W_B),\M_0)$.
If $p$ and $n$ are coprime then  $H^1(SL_n(W_B),\M_0)=(0)$ and we can  find $X\in SL_n(A)$ with $\pi(X)=I$ such that $XGX^{-1}=SL_n(W_A)$. When $p$ divides $n$ we  use $H^1(SL_n(W_B),\M)=(0)$ to find $X\in GL_n(A)$ with $\pi(X)=I$ such that $XGX^{-1}=SL_n(W_A)$.

We now consider the case when
\begin{equation}\label{seq 2}0\to\s\to G\to SL_n(W_B)\to I\end{equation}
 is exact.  Since $H^2(SL_n(W_B),\s)\to H^2(SL_n(W_B),\M_0)$  is injective
by Proposition \ref{SL_n(W_m)}, part v(c), the sequence (\ref{seq 2})
splits if and only if the sequence (\ref{seq 1}) splits. We cannot have $W_A=W_{m+1}$ and $W_B=W_m$ because that will contradict  the non-splitting of
 extension (\ref{extn1}). Thus  $W_A=W_B$, the sequences (\ref{seq 1}) and
  (\ref{seq 2}) split, and we are in the set up covered by the second part.
\end{proof}

\begin{remark} One can check that except for Theorem 3.4 of  \cite{manoharmayum}---which didn't influence
the argument there anyway---all the results in sections 3 and 4 of \cite{manoharmayum}
remain valid even when $\pmb{k}$ and $n$ satisfy the assumptions of (\ref{setup}). In these cases, the desired conclusions all follow from the results of Sah (\cite{sah1}, \cite{sah2})
covered in Proposition \ref{SL_n(W_m)}, either directly or with a small line of argument. We
leave the precise verification to the interested reader,  and state the following
extension of the main theorem of \cite{manoharmayum}.

\begin{theorem}\label{extended structure thm}
Let $(A,\m_A)$ be a complete local noetherian ring with maximal ideal
$\m_A$ and finite residue field $A/\m_A$ of characteristic $p$. Suppose we are given a subfield $\pmb{k}$ of $A/\m_A$ and a closed
 subgroup  $G$ of $GL_n(A)$ such that
 \begin{itemize}
  \item $n\geq 2$ and the pair $(n,|\pmb{k}|)$ is not one of the following:
 $(2,2)$, $(2,3)$, $(2,5)$, $(3,2)$, or  $(4,2)$;
 \item $G\mod{\m_A}\supseteq SL_n(\pmb{k})$. \end{itemize}
Then
     $G$ contains a conjugate of $SL_n(W_A)$.
\end{theorem}

The restrictions on $(n,|\pmb{k}|)$ are necessary.  It is  known (see \cite{atlas}) that $SL_4(\F_2)$ has a double cover inside $GL_4(\Z/4\Z)$, and from the orders of the groups we see that this double cover cannot contain a conjugate of $SL_4(\Z/4\Z)$. The case when $(3,\F_2)$ must be an exception because extension 
$I\to \Gamma_1\to SL_3(\Z/4\Z)\to SL_3(\F_2)\to I$ splits; the other exceptions when $n=2$ are covered in \cite{manoharmayum}.
\end{remark}

\bibliographystyle{abbrv}

\end{document}